\documentclass{amsart}
\usepackage{amsmath, amsthm, amscd, amssymb,latexsym,enumerate}
\usepackage{url}
\usepackage{array}
\usepackage[all]{xy}
\usepackage{enumitem}
\usepackage{hyperref}

\def\Idem{\mathrm{Id}}

\def\Rhom{\mathrm{Rhom}}
\def\B{\mathbf{B}}
\def\Z{\mathbb{Z}}
\def\Q{\mathbb{Q}}
\def\F{\mathbb{F}}
\def\R{\mathbb{R}}
\def\C{\mathbb{C}}

\def\pp{{\mathfrak{p}}}
\def\qq{{\mathfrak{q}}}
\def\rr{{\mathfrak{r}}}

\def\h{{h}}

\def\CC{{\mathbf C}}

\def\rad{\mathrm{rad}}

\def\sep{\mathrm{sep}}

\def\Hom{\mathrm{Hom}}

\def\dim{\mathrm{dim}}

\def\lcm{\mathrm{lcm}}

\def\rank{\mathrm{rank}}
\def\I{{\mathcal I}}
\def\Id{\mathrm{Id}}
\def\ord{\mathrm{ord}}

\def\F{{\mathbb F}}

\def\isom{\xrightarrow{\sim}}

\numberwithin{equation}{section}

\newtheorem{thm}{Theorem}
\numberwithin{thm}{section}

\newtheorem{lem}[thm]{Lemma}
\newtheorem{cor}[thm]{Corollary}
\newtheorem{prop}[thm]{Proposition}

\theoremstyle{definition}
\newtheorem{defn}[thm]{Definition}

\newtheorem{ex}[thm]{Example}
\newtheorem{exs}[thm]{Examples}
\newtheorem{rem}[thm]{Remark}
\newtheorem{rems}[thm]{Remarks}

\numberwithin{equation}{thm}

\title
[Universal gradings of orders]
{Universal gradings of orders}
\author[H.\ W.\ Lenstra, Jr.]{H.\ W.\ Lenstra, Jr.}
\address{Mathematisch Instituut, Universiteit Leiden, The Netherlands}
\email{hwl@math.leidenuniv.nl}
\author[A.\ Silverberg]{A.\ Silverberg}
\address{Department of Mathematics, University of California, Irvine, CA 92697, USA}
\email{asilverb@uci.edu}
\thanks{Support for the research was provided by the Alfred P.~Sloan Foundation.
}

\begin{document}

\begin{abstract} 
For commutative rings, we introduce the notion of a {\em universal grading}, which can be viewed as the ``largest possible grading''. While not every commutative ring (or order) has a universal grading, we prove that every {\em reduced order} has a universal grading, and this grading is by a {\em finite} group. Examples of graded orders are provided by group rings of finite abelian groups over rings of integers in number fields. We generalize known properties of nilpotents, idempotents, and roots of unity in such group rings to the case of graded orders; this has applications to cryptography. Lattices play an important role in this paper; a novel aspect is that our proofs use that the additive group of any reduced order can in a natural way be equipped with a lattice structure.
\end{abstract}

\maketitle

\section{Introduction}

In 1940, G. Higman \cite[Theorem 3]{Higman} proved the beautiful result
that if $\Gamma$ is a finite abelian group, then the torsion subgroup
of the group of units of the group ring $\Z[\Gamma]$ equals
$\pm\Gamma$. His proof was remarkable in that it depended on
properties of the absolute value of complex numbers.

In recent work \cite{CMorders} on cryptography, the present authors needed to
use a similar result on rings that are a bit more general than
Higman's group rings, namely {\it graded orders}. Here an 
{\bf order} is a commutative ring $A$ of which the additive group $A^+$
is isomorphic to $\Z^n$ for some $n\in\Z_{\ge0}$, and {\it
graded\/} refers to the familiar notion recalled below; our
gradings will always be by abelian groups. If the order $A$ is
{\bf reduced} in the sense that its nilradical is $0$, then the
group $A^+$ carries a natural {\it lattice structure}. Replacing
Higman's technique by this lattice structure, we were able to
prove basic properties of nilpotents, idempotents, and torsion
units in any graded order, as expressed in Theorem \ref{firstthm} below.

Much to our surprise, we discovered that the same lattice
structure can be used to prove a far more fundamental result on
graded orders. Namely, as our main theorem (Theorem \ref{univgradingorderthm})
asserts, each reduced order $A$ has a {\it universal grading},
which controls {\it all\/} gradings of $A$ and can be thought of
as its ``finest possible'' grading. 
The precise definition is given in Definition \ref{univgradingdefn} below. 
This definition does not appear to occur in the literature, presumably because prior to our discovery no interesting class of examples was known; and 
indeed, many naturally occurring rings fail to have universal gradings.

Our main result suggests a number of promising avenues for further
research. The first is to exhibit a larger class of commutative
rings that have universal gradings. For Higman's original result,
several far-reaching generalizations have been found, notably in
the work of W.~May \cite{May}. Replacing our ``archimedean'' arguments
by arguments with a $p$-adic flavor, one can probably identify
algebraic conditions that ensure the existence of a universal
grading.

Secondly, we hope to show in forthcoming work \cite{GpRings} that the
existence of a universal grading on any reduced order has
important consequences for the problem of how one may write a given
commutative ring as a group ring, a problem that is closely
related to the well-studied subject of isomorphisms between group
rings. Roughly speaking, we prove that, up to isomorphism, there
is a unique ``maximal'' way of realizing a given reduced order as
a group ring. 
Such results are probably also achievable over more general base
rings than the ring of integers.

Third, there is the algorithmic question of designing an
``efficient'' method for computing the universal grading of a
given reduced order, see \cite{DvG}.

Fourth, our main result may be rephrased by saying that there is, in a suitable sense, a ``maximal'' abelian group scheme ``of multiplicative type'' that acts on a given reduced order (see \cite{Oesterle}).
One may wonder whether a similar result
holds for more general finite abelian group schemes. 

In this paper all rings are supposed to be commutative.

\begin{defn}
\label{gradingdefn}
Suppose $\Gamma$ is a multiplicatively written abelian group
with identity element $1$. 
Then a {\bf $\Gamma$-grading} of a ring $A$
is a system $\B = (B_\gamma)_{\gamma\in \Gamma}$ of additive subgroups
$B_\gamma \subset A$  that satisfies:
\begin{enumerate}
\item 
$B_\gamma \cdot B_{\gamma'} \subset B_{\gamma\gamma'}$
for all $\gamma,\gamma'\in \Gamma$, and
\item
$A = \bigoplus_{\gamma\in \Gamma} B_\gamma$
in the sense that the additive group homomorphism  
$\bigoplus_{\gamma\in \Gamma} B_\gamma \to A$ sending 
$(x_\gamma)_{\gamma\in \Gamma}$ to $\sum_{\gamma\in \Gamma}x_\gamma$
is bijective.
\end{enumerate}
\end{defn}

We note that if $R$ is a ring and $\Gamma$ is an abelian group, then
there is a natural $\Gamma$-grading of the group ring $R[\Gamma]$,
given by $(R\cdot\gamma)_{\gamma\in \Gamma}$.

If  $f : \Gamma \to \Delta$ is a
homomorphism of abelian groups, then each $\Gamma$-grading 
$\B = (B_\gamma)_{\gamma\in \Gamma}$ of a ring $A$
gives rise to the $\Delta$-grading 
$(\sum_{\gamma\in f^{-1}(\delta)}B_\gamma)_{\delta\in \Delta}$
of $A$, which we denote by $f_\ast\B$.

\begin{defn}
\label{univgradingdefn}
By a {\bf universal grading} of a ring $A$ we mean a pair
$(\Gamma,\B)$ consisting of an abelian group $\Gamma$ and
a $\Gamma$-grading $\B$ of $A$ with the property that for each
abelian group $\Delta$ and each $\Delta$-grading $\CC$ of $A$
there is a unique group homomorphism
$f : \Gamma \to \Delta$ such that $\CC = f_\ast\B$.
\end{defn}

If a universal grading of $A$ exists, then by a standard argument
it is, in an obvious sense, unique up to a unique isomorphism;
and it exists if and only if the functor
that assigns to an abelian group $\Delta$ the set of 
$\Delta$-gradings of $A$ is representable.

Many naturally occurring rings fail to have a universal grading;
see Examples \ref{nounivgradexs}(i,ii,iii)  
for number fields and finite fields that have no universal grading.
This makes the following result all the more unexpected.

\begin{thm}
\label{univgradingorderthm}
Every reduced order has a universal grading, and 
its universal grading is by a finite abelian group.
\end{thm}

We prove Theorem \ref{univgradingorderthm} in section \ref{Thm11pf}
(using lemmas given earlier in the paper).

It could be of interest to study non-reduced orders as well. 
In Examples~\ref{nounivgradexs}(vi--viii)
we show that they may have a universal grading by an infinite group, 
or by a finite group, 
or no universal grading at all.
In particular, one cannot omit ``reduced'' from Theorem \ref{univgradingorderthm}.

In section \ref{noncyclicthmpf} we prove the following result,
which answers a question posed by Kiran Kedlaya.

\begin{thm}
\label{noncyclicthm}
Let $A$ be an order that is a Dedekind domain. Then the universal grading of $A$ is by a finite cyclic group.
\end{thm}

Suppose $A$ is a ring.  
The set of nilpotent elements of $A$ is an ideal of $A$, 
denoted $\sqrt{0}$ or $\sqrt{0_A}$
and called the {\bf nilradical}. 
We call $x \in A$ an {\bf idempotent} if $x^2 = x$.
We denote the set of idempotents by $\Idem(A)$,
and we call
$A$  {\bf connected} if $\#\Idem(A) = 2$ or, equivalently,
if one has $\Id(A) = \{ 0,1\}$ and $A\neq 0$.
We call $x\in A$ a {\bf root of unity} if $x^n=1$ for some $n\in\Z_{>0}$.
The set of roots of unity of $A$, which is a subgroup of the
group $A^\ast$ of units of $A$, is denoted by $\mu(A)$.

Let $A$ be a ring and let $(B_\gamma)_{\gamma\in \Gamma}$
be a  $\Gamma$-grading of $A$.
Then the subgroup $B_1$ of $A$ is a subring of $A$
that contains the identity element
of $A$ (see Lemma \ref{B1ring}).
We shall call an additive subgroup $H \subset A$
{\bf homogeneous} if for each
$(x_\gamma)_{\gamma\in \Gamma} \in \bigoplus_{\gamma\in \Gamma} B_\gamma$
one has that 
$\sum_{\gamma\in \Gamma}x_\gamma$ is in $H$ if and only if
each $x_\gamma$ is in $H$ (i.e., 
$H = \bigoplus_{\gamma\in \Gamma} (H \cap B_\gamma)$ via the
bijection in (ii) above).
This terminology will in particular be applied to ideals
and to subrings of $A$.
An  element of $A$ is called {\bf homogeneous} if
it belongs to $\bigcup_{\gamma\in \Gamma} B_\gamma$.

\begin{thm}
\label{firstthm}
Let $\Gamma$ be an abelian group,
and let $A$ be an order with $\Gamma$-grading
$(B_\gamma)_{\gamma\in \Gamma}$. Then:
\begin{enumerate}
\item 
the nilradical $\sqrt{0_A}$ is a homogeneous ideal of $A$;
\item
$\Idem(A) = \Idem(B_1)$, and $A$ is
connected if and only if $B_1$ is connected;
\item
if $B_1$ is connected, then each element of $\mu(A)$ is homogeneous.
\end{enumerate}
\end{thm}

The three parts of Theorem \ref{firstthm} are proved in
Propositions \ref{AGamma}(iii), \ref{idemsequal},  and \ref{connhomogprop},
respectively.
Note that Theorem \ref{firstthm}(iii) is clearly false if the connectedness assumption is dropped.

In the case that $A$ is a group ring $B[\Gamma]$ with its natural $\Gamma$-grading,
with $B$ an order and $\Gamma$ a finite abelian group,
Theorem \ref{firstthm} 
was known and can be deduced 
from results in \cite{May}
(Proposition 2 of  \cite{May} for (i), 
the Corollary to Proposition 3 for (ii), 
and the Corollary to Proposition 10 for (iii)).

We end the introduction with two important classes of examples of graded rings.
\begin{ex}[{\em Kummer extensions}]
Let $K\subset L$ be a field extension, and let $W$ be the set of $a\in L^\ast$ for which there exists $n \in \Z_{>0}$ such that $a^n\in K^\ast$ and $K$ contains a primitive $n$-th root of unity. Then $W$ is a subgroup of $L^\ast$ containing $K^\ast$, and the subfield $K(W)$ of $L$ is graded by the group $W/K^\ast$; here the piece of degree $aK^\ast\in  W/K^\ast$ is the one-dimensional $K$-vector space $Ka$. 
This example illustrates that finding a grading for a field extension is closely related to the classical problem of generating the field by means of radicals. 
\end{ex}

\begin{ex}[{\em Extended tensor algebras}]
Suppose $A$ is a commutative ring and $L$ is a projective $A$-module of rank 1. 
For $i\in \Z$, let $L^{{\otimes}i}$ denote the $i$-th tensor power of $L$, 
where for negative values of $i$ 
we define $L^{{\otimes}i}=\Hom_A(L^{{\otimes}-i},A)$. 
Then the extended tensor algebra 
$\Lambda=\bigoplus_{i \in \Z}L^{{\otimes}i}$ is graded by an infinite cyclic group.
If $r\in \Z_{>0}$ and $L^{{\otimes}r}$ is free, say $L^{\otimes r}=Ay$, then the ring 
$B=\Lambda/(y-1)\Lambda$ is graded by a cyclic group of order $r$, since 
$B=\bigoplus_{i=0}^{r-1}L^{{\otimes}i}$. 
This class of examples includes the graded orders that we encountered in 
lattice-based cryptography,
and that play crucial roles in the proofs of the main results in \cite{LwS,CMorders}.
More precisely,
Theorem \ref{firstthm}(ii,iii) 
supplies the proof of Proposition 14.3(iv) of \cite{CMorders}.
\end{ex}

\subsection*{Acknowledgments}
We thank Warren May for providing the references to \cite{May},
Bas Edixhoven for helpful comments,
Daan van Gent for Example \ref{nounivgradexs}(i),
and Kiran Kedlaya for helpful comments and
for posing a question that led to Theorem \ref{noncyclicthm}.

\section{Graded rings}

In this section we give some relatively straightforward
lemmas that we will use to prove our main results.
The proofs of Theorems \ref{univgradingorderthm}
and \ref{firstthm} depend on two techniques. 
One, mentioned earlier, depends
on the introduction of a natural lattice structure on any
reduced order.
The other (Lemma \ref{fchithm} below) consists of equipping a 
$\Gamma$-graded ring with an action by the dual of $\Gamma$,
after a suitable cyclotomic base change;
here $\Gamma$ is finite.

\begin{lem}
\label{B1ring}
Suppose $A$ is a  
ring, $\Gamma$ is an abelian group, and 
$(B_\gamma)_{\gamma\in \Gamma}$ is a $\Gamma$-grading of $A$.
Then: 
\begin{enumerate}
\item 
$1\in B_1$,
\item
$B_1$ is a ring, and 
\item
each $B_\gamma$ is a $B_1$-module.
\end{enumerate}
\end{lem}

\begin{proof}
Write $1=(1_\gamma)_{\gamma\in \Gamma} \in A$. 
Take any $\delta\in \Gamma$ and $\alpha\in B_\delta$.
Then
$
\alpha = 1\cdot\alpha = 
(1_\gamma)_{\gamma\in \Gamma}\cdot(\alpha_\gamma)_{\gamma\in \Gamma}
$
where $\alpha_\delta = \alpha$ and $\alpha_\gamma = 0$
for all $\gamma \neq \delta$.
Comparing $\delta$-coordinates we have
$\alpha = 1_1\cdot\alpha$, and likewise
$\alpha = \alpha\cdot 1_1$.
So $1_1$ acts left and right as the identity on each $B_\delta$,
and hence on $A$. Thus, $1 = 1_1 \in B_1$, proving (i).
Parts (ii) and (iii) are straightforward.
\end{proof}

If $\Gamma$ is an abelian group and $k\in\Z$, let 
$\Gamma^k = \{ \gamma^k : \gamma \in \Gamma\}$.
The following two lemmas will be used to prove Proposition \ref{Esephomog} and
Proposition \ref{redordercor}, respectively.

\begin{lem}
\label{gammamischomog}
Suppose 
$\Gamma$ is an abelian group,  
$\B = (B_\gamma)_{\gamma\in \Gamma}$ is a $\Gamma$-grading of a commutative ring $A$,
and the set $S = \{ \gamma\in \Gamma : B_\gamma \neq 0 \}$ is finite.
Then there are a finite abelian group $\Delta$ and a
$\Delta$-grading $\CC = (C_\delta)_{\delta\in \Delta}$ of $A$ such that
$\bigcup_{\gamma\in \Gamma}B_{\gamma} =
\bigcup_{\delta\in \Delta}C_{\delta}.$
\end{lem}

\begin{proof}
We can and do replace $\Gamma$ with $\langle S \rangle$.
Since $\{ 1\} = \bigcap_{N\in\Z_{>0}}\Gamma^N$, if
$s,t\in S$ with $s\neq t$ then there exists $N_{s,t}\in\Z_{>0}$
such that $st^{-1}\notin \Gamma^{N_{s,t}}$.
Let $M = \lcm_{s,t\in S, s\neq t}\{ N_{s,t}\}$,
let $c : \Gamma \to \Gamma/\Gamma^M$ be the canonical projection map,
and let $\CC = c_\ast\B = (C_\delta)_{\delta\in \Gamma/\Gamma^M}$.
By construction, the restriction of $c$ to $S$ is injective,
and the desired result now follows with $\Delta=\Gamma/\Gamma^M$.
\end{proof}

\begin{lem}
\label{gammamisc}
Suppose $A$ is a commutative ring,
$\Gamma$ is an abelian group,  
$\B = (B_\gamma)_{\gamma\in \Gamma}$ is a $\Gamma$-grading of $A$,
and $(\Gamma,\B)$ is universal.
Then $\Gamma = \langle \gamma\in \Gamma : B_\gamma \neq 0 \rangle$.
\end{lem}

\begin{proof}
Put $\Delta = \Gamma/\langle \gamma\in \Gamma : B_\gamma \neq 0 \rangle$,
and let $t,c : \Gamma \to \Delta$ be the trivial and the canonical map,
respectively. Then $t$ and $c$ agree on each $\gamma$ with $B_\gamma \neq 0$,
so $t_\ast \B = c_\ast \B$, and by universality one gets $t=c$ so $\Delta = \{ 1\}$.
\end{proof}

We will use the next lemma to prove Lemma \ref{redorderlem} and
Proposition \ref{Esephomog}.

\begin{lem}
\label{orderlem}
Suppose $\Gamma$ is an abelian group, $A$ is either
a commutative $\Q$-algebra with $\dim_\Q A < \infty$ or an order, and 
$(B_\gamma)_{\gamma\in \Gamma}$ is a $\Gamma$-grading of $A$.
Then $B_\gamma = 0$ for all but finitely many $\gamma\in \Gamma$.
\end{lem}

\begin{proof}
This holds since $A = \bigoplus_{\gamma\in \Gamma} B_\gamma$,
and $A$ has finite $\Z$-rank (if $A$ is an order) or  finite $\Q$-dimension 
(if $A$ is a finite dimensional commutative $\Q$-algebra).
\end{proof}

Suppose $k\in\Z_{>0}$. With
$\Phi_k$ denoting the $k$-th cyclotomic polynomial and $\zeta_k = X+(\Phi_k)$,
we have 
$
\Z[\zeta_k] = \Z[X]/(\Phi_k) = \bigoplus_{i=0}^{\varphi(k)-1} \Z\cdot\zeta_k^i,
$ where $\varphi$ is the Euler $\varphi$-function.
Suppose $A$ is a  
ring, $\Gamma$ is an abelian group,  and
$(B_\gamma)_{\gamma\in \Gamma}$ is a $\Gamma$-grading of $A$.
Then
$B_\gamma[\zeta_k] = B_\gamma\otimes_\Z \Z[\zeta_k]$ is a module over $B_1[\zeta_k]$ 
for all $\gamma\in \Gamma$, and
$
A[\zeta_k] = A\otimes_\Z \Z[\zeta_k] = 
\bigoplus_{\gamma\in \Gamma} (B_\gamma[\zeta_k])
$
is a $\Gamma$-graded ring that contains $A$.
If $\Gamma$ is finite of exponent dividing $k$, 
we let 
$$
\hat{\Gamma}_k = \Hom(\Gamma,\langle \zeta_k \rangle),
$$
a multiplicative group with $\#\hat{\Gamma}_k = \#\Gamma$.
We use the next lemma to prove Propositions \ref{AGamma} and \ref{Aorthog}.

\begin{lem}
\label{fchithm}
Suppose $A$ is a  
ring, $\Gamma$ is a finite abelian group,  
$(B_\gamma)_{\gamma\in \Gamma}$ is a $\Gamma$-grading of $A$, 
and 
$k$ is a positive integer divisible by the exponent of $\Gamma$.
For $\chi\in\hat{\Gamma}_k$, and 
$\alpha = (\alpha_\gamma)_{\gamma\in \Gamma} \in A[\zeta_k]$
with $\alpha_\gamma \in B_\gamma[\zeta_k]$, define
$$
\chi\ast\alpha = (\chi(\gamma)\cdot\alpha_\gamma)_{\gamma\in \Gamma} \in A[\zeta_k].$$
This defines an action of $\hat{\Gamma}_k$ on $A[\zeta_k]$ by
ring automorphisms, and for all $\delta\in \Gamma$ 
and $\alpha = (\alpha_\gamma)_{\gamma\in \Gamma} \in A[\zeta_k]$
one has
$$
\sum_{\chi\in\hat{\Gamma}_k} \chi\ast(\chi(\delta)^{-1}\alpha) = 
\#\Gamma\cdot \alpha_\delta
\in B_\delta[\zeta_k] \subset A[\zeta_k].
$$
\end{lem}

\begin{proof}
The proof is an easy exercise. 
The last statement follows from the fact that if $\delta \in \Gamma$ then
$\sum_{\chi\in\hat{\Gamma}_k} \chi(\delta)$ is 
$\#\Gamma$ if $\delta = 1$, and otherwise is $0$.
\end{proof}

\section{Euclidean vector spaces, lattices, and orders}

In a series of examples, we introduce the lattice structure on
reduced orders that we will use in the proofs of our main results.
We conclude the section with a result on gradings of reduced orders
that we will use to prove Proposition \ref{Aorthog} and Theorem \ref{univgradingorderthm}.

If $C$ is a $\Z$-module or $\Z$-algebra, we will write $C_\Q$ for $C\otimes_\Z\Q$.

A {\bf Euclidean vector space} is a finite dimensional $\R$-vector space
$E$ equipped with a map 
$\langle \, \, \, , \, \,  \,  \rangle : E \times E \to \R$,
$(x,y)\mapsto \langle x,y \rangle$ that is
$\R$-bilinear, symmetric, and positive definite.

\begin{ex}
\label{Euclideaneex}
Suppose $E$ is a finite dimensional $\R$-vector space
equipped with a map 
$\langle \, \, \, , \, \,  \,  \rangle : E \times E \to \R$ that is
$\R$-bilinear, symmetric, and positive semidefinite.
Let
$$
\rad(E) = \{ x\in E : \langle x,E \rangle = 0\}.
$$
Then 
$\rad(E)  = \{ x\in E : \langle x,x \rangle = 0\}$,
and $\langle \,  \, , \,   \,  \rangle$ makes
$E/\rad(E)$ into a Euclidean vector space.
\end{ex}

\begin{ex}
\label{Elatticeex}
Suppose $E$ is a commutative $\R$-algebra with $\dim_\R(E) < \infty$.
For all $x,y\in E$, let 
$
\langle x,y \rangle = \sum_{\sigma : E \to\C} \sigma(x)\overline{\sigma(y)},
$
where $\sigma$ ranges over all $\R$-algebra homomorphisms from $E$ to $\C$.
Then $\rad(E) = \sqrt{0_E}$.
(If $x\in\sqrt{0_E}$ then $\sigma(x)=0$ for all $\sigma$, so 
$\langle x,y\rangle = 0$ for all $y$, so $x\in\rad(E)$.
Conversely, $E/\sqrt{0_E}$ is a product of fields, and these fields
are 
$\R$ and $\C$. Since the inner products on $\R$ and $\C$ are positive
definite, so is the inner product on $E$. Thus $\rad(E/\sqrt{0_E}) =0$,
so $\rad(E) \subset \sqrt{0_E}$. Note that, as a consequence, the number of $\sigma$'s equals $\dim_\R(E)$.)
\end{ex}

Recall that a  
{\bf lattice} is a finitely generated
free abelian group $L$ equipped with a positive definite symmetric
$\R$-bilinear function $\langle \,\, ,  \,\, \rangle : L_\R \times L_\R \to \R$,
where $L_\R = L\otimes_\Z\R$.

If $B$ and $C$ are rings, we write $\Rhom(B,C)$ for the
set of ring homomorphisms from $B$ to $C$.

\begin{ex}
\label{orderlatticeex}
Suppose $A$ is an order. Then $E = A_\R$
is a finite dimensional $\R$-vector space
equipped with an $\R$-bilinear, symmetric, positive semidefinite inner product 
$\langle \,\, ,  \,\, \rangle : E \times E \to \R$
as in Example \ref{Elatticeex}.
Further, 
$\rad(E) = \sqrt{0_E} = (\sqrt{0_A})_\R$, and thus
$A/\sqrt{0_A}$ has a natural lattice structure.
(That $(\sqrt{0_A})_\R \subset \sqrt{0_E}$ is clear.
For the reverse inclusion, $A/\sqrt{0_A}$ is a reduced order,
so $(A/\sqrt{0_A})_\Q$ is a product of finitely many number fields,
so is a product of finitely many separable extensions of $\Q$.
It follows that $(A/\sqrt{0_A})_\R = E/(\sqrt{0_A})_\R$ is a product of 
finitely many separable extensions of $\R$, so is reduced.
It also follows that $\#\Rhom(A,\C)$ equals 
$\rank(A/\sqrt{0_A})$, the rank of $A/\sqrt{0_A}$ as an abelian group.)
\end{ex}

\begin{ex}
\label{orderlatticeex4}
Suppose $A$ is a reduced order.
Then $A/\sqrt{0_A}=A$, so by the previous example $A$ has a natural lattice structure. It is given by
$$
\langle x,y \rangle  
= \sum_{\sigma\in\Rhom(A,\C)} \sigma(x)\overline{\sigma(y)}
$$
for $x$, $y\in A$.
Note that $\#\Rhom(A,\C)=\rank(A)$. It follows that one has
\begin{equation}
\label{zetarkeqn}
\langle\zeta,\zeta\rangle=\rank(A)\qquad\hbox{for every\ }\zeta\in\mu(A).
\end{equation}
\end{ex}

\begin{ex}
\label{orderlatticeex5}
Let $\Gamma$ be a finite abelian group, and let $A=\Z[\Gamma]$. 
A short computation shows that for 
$x=\sum_{\gamma\in\Gamma}x_\gamma\gamma$ (with $x_\gamma\in\Z$) one has
$$\langle x,x\rangle=\#\Gamma\cdot\sum_{\gamma\in\Gamma}x_\gamma^2.$$
Hence for $x\ne 0$ one has $\langle x,x\rangle\ge\#\Gamma$, with equality if and only if $x\in\pm\Gamma$. Combining this with  \eqref{zetarkeqn}, one obtains Higman's theorem $\mu(\Z[\Gamma])=\pm\Gamma$.
\end{ex}

\begin{ex}
\label{orderlatticeex6}
Let $\Gamma$ be a finite abelian group, let $I$ be the $\Z[\Gamma]$-ideal 
$\Z\cdot\sum_{\gamma\in\Gamma}\gamma$, and put $A=\Z[\Gamma]/I$. For $x=(\sum_{\gamma\in\Gamma}x_\gamma\gamma)+I\in A$ (with $x_\gamma\in\Z$), one computes
$$\langle x,x\rangle=\sum_{{\gamma,\delta\in\Gamma}\atop{\gamma<\delta}}(x_\gamma-x_\delta)^2,$$
where $<$ is any total ordering on $\Gamma$. One readily deduces that for $x\ne0$ this is at least $\#\Gamma-1=\rank(A)$, with equality if and only if $x\in\pm\Gamma+I$. As before, one deduces $\mu(\Z[\Gamma]/I)=\pm\Gamma+I$.
\end{ex}

\begin{ex}
\label{orderlatticeex7}
Contrary to what the previous two examples might suggest, 
it is not the case that 
$\langle x,x\rangle\ge\rank(A)$
for every non-zero $x$ in a reduced order $A$, 
not even when $A$ is connected. 
For example, let $A$ be the subring of the product ring 
$\Z\times\Z\times\Z\times\Z\times\Z$ 
consisting of those 
elements whose coordinates have the same parity, and 
choose $x=(2,0,0,0,0)$. 
Then $\rank(A)=5$ and $\langle x,x\rangle = 4$.
We will refer to this example in Remark \ref{exampremrtunity},
concerning the proof of Theorem \ref{firstthm}(iii).
\end{ex}

\begin{lem}
\label{redorderlem}
Suppose $\Gamma$ is an abelian group, 
$A$ is either
a commutative $\Q$-algebra with $\dim_\Q A < \infty$ or an order,  
$(B_\gamma)_{\gamma\in \Gamma}$ is a $\Gamma$-grading of $A$,
and $A$ has no non-zero homogeneous nilpotent elements.
Then:
\begin{enumerate}
\item
if $\delta\in\Gamma$ and $\delta$ has infinite order, then $B_\delta = 0$;
\item
the subgroup $\langle \gamma\in \Gamma : B_\gamma \neq 0 \rangle$
is finite.
\end{enumerate}
\end{lem}

\begin{proof}
By Lemma \ref{orderlem},
for all but finitely many $\gamma\in \Gamma$ we have 
$B_\gamma = 0$.
Suppose $\delta\in\Gamma$ has infinite order.
Then there exists $N\in\Z_{>0}$ such that 
$B_{\delta^N} = 0$.
Suppose $x\in B_\delta$.
Then $x^N \in (B_{\delta})^N \subset B_{\delta^N} = 0$,
so $x$ is homogeneous and nilpotent. By our assumption, 
$x=0$, proving (i).
Thus the abelian group
$\langle \gamma\in \Gamma : B_\gamma \neq 0 \rangle$ is generated by 
finitely many elements of finite order,  
so this group is finite, proving (ii).
\end{proof}

\begin{prop}
\label{redordercor}
Suppose $\Gamma$ is an abelian group, $A$ is a reduced order, and 
$\B = (B_\gamma)_{\gamma\in \Gamma}$ is a $\Gamma$-grading of $A$.
Then:
\begin{enumerate}
\item
the subgroup $\langle \gamma\in \Gamma : B_\gamma \neq 0 \rangle$
is finite;
\item
if $(\Gamma,\B)$ is universal, then $\Gamma$ is finite.
\end{enumerate}
\end{prop}

\begin{proof}
Since $A$ is reduced, it has no non-zero nilpotent elements, so
(i) follows from Lemma \ref{redorderlem}(ii).
Part (ii) now follows from (i) and Lemma \ref{gammamisc}.
\end{proof}

\section{Nilpotent and separable elements}

We next prove Theorem \ref{firstthm}(i).
If $R$ is a ring and $m\in\Z_{>0}$,
we write $R^+[m]$ for the $m$-torsion in the
additive group $R$.

\begin{prop}
\label{AGamma}
Suppose $A$ is a  
ring, $\Gamma$ is an abelian group, and 
$(B_\gamma)_{\gamma\in \Gamma}$ is a $\Gamma$-grading of $A$. 
\begin{enumerate}
\item
If $\Gamma$ is finite
and $\alpha = (\alpha_\gamma)_{\gamma\in \Gamma} \in \sqrt{0_A}$,
then $\# \Gamma\cdot\alpha_\delta \in \sqrt{0_A}$ for all $\delta\in \Gamma$.
\item
If $\Gamma$ is finite
and  $A^+[\# \Gamma] =0$,
then $\sqrt{0_A}$ is a homogeneous ideal.
\item
If $A$ is an order,
then $\sqrt{0_A}$ is a homogeneous ideal.
\end{enumerate}
\end{prop}

\begin{proof}
We first prove (i).
Let $k$ denote the exponent of the finite group $\Gamma$ and let
$A' = A[\zeta_k]$.
We have $\alpha \in \sqrt{0_A} \subset \sqrt{0_{A'}}$, and
since $\sqrt{0_{A'}}$ is an ideal we have
$\chi(\delta)^{-1}\alpha \in \sqrt{0_{A'}}$
for all $\chi\in\hat{\Gamma}_k$ and $\delta\in\Gamma$.
Since $\hat{\Gamma}_k$ acts by ring automorphisms (Lemma \ref{fchithm}), we have
$\sum_{\chi\in\hat{\Gamma}_k} \chi\ast(\chi(\delta)^{-1}\alpha) \in \sqrt{0_{A'}}$
for all $\delta\in\Gamma$.
By Lemma~\ref{fchithm} we now have 
$\# \Gamma\cdot\alpha_\delta \in \sqrt{0_{A'}} \cap A =  \sqrt{0_A}$
for all $\delta\in\Gamma$.

We next prove (ii).
Clearly, $\bigoplus_{\gamma\in \Gamma} (\sqrt{0_A} \cap B_\gamma) \subset \sqrt{0_A}$.
For the reverse inclusion, suppose 
$\alpha = (\alpha_\gamma)_{\gamma\in \Gamma} \in \sqrt{0_A}$ and $\delta\in \Gamma$.
By (i) we have 
$(\# \Gamma\cdot \alpha_\delta)^N = 0$ for some $N\in\Z_{>0}$.
But $(\# \Gamma\cdot \alpha_\delta)^N = (\# \Gamma)^N\alpha_\delta^N$.
If $A^+[\# \Gamma] =0$, then $\alpha_\delta^N = 0$,
so $\alpha_\delta \in \sqrt{0_A}$ as desired.

For (iii), let $\I$ denote the ideal generated by the 
homogeneous nilpotent elements of $A$, i.e.,
$\I$ is the largest homogeneous ideal of $A$ contained in $\sqrt{0_A}$.
Then $A/\I$ has  a $\Gamma$-grading $(C_\gamma)_{\gamma\in \Gamma}$
with $C_\gamma = B_\gamma/(\sqrt{0_A}\cap B_\gamma)$,
and $A/\I$ is an order with no non-zero homogeneous nilpotent elements.
By Lemma \ref{redorderlem}(ii), the subgroup
$\langle \gamma\in \Gamma : C_\gamma \neq 0 \rangle$
is finite;
we can and do replace $\Gamma$ with this finite group.
Since  orders have no non-zero torsion, (iii) now follows from (ii). 
\end{proof}

The following example shows that the condition that $A^+[\# \Gamma] =0$
cannot be dropped from Proposition \ref{AGamma}(ii).

\begin{ex}
Suppose $p$ is a prime number 
and
$\Gamma$ is any finite abelian group of order divisible by $p$.
Then $A = \F_p[\Gamma] = \bigoplus_{\gamma\in \Gamma} \F_p\cdot\gamma$
is a $\Gamma$-graded ring and 
$(\sum_{\gamma\in \Gamma} \gamma)^2 = \# \Gamma\sum_{\gamma\in \Gamma}\gamma = 0$.
So $\sum_{\gamma\in \Gamma}\gamma\in \sqrt{0_A}$,  
but the coordinates
$\gamma$ of $\sum_{\gamma\in \Gamma}\gamma$ 
are units and thus are not nilpotent, so the ideal $\sqrt{0_A}$
is not homogeneous.
\end{ex}

We call a polynomial $f\in\Q[X]$ {\bf separable}
if $f$ is coprime to its derivative $f'$.
If $E$ is a commutative $\Q$-algebra
with $\dim_\Q E < \infty$, then $\alpha\in E$ is called 
{\bf separable} if there exists a separable polynomial
$f\in\Q[X]$ with $f(\alpha)=0$.
We write $E_\sep$ for the set of separable elements of $E$.
Note that $E_\sep$ is a sub-$\Q$-algebra of $E$ (see for example
Lemma 2.2 of \cite{Qalgs}).
We will use the next result to prove Theorem \ref{firstthm}(iii).

\begin{prop}
\label{Esephomog}
If $\Gamma$ is an abelian group
and $E = \bigoplus_{\gamma\in \Gamma} E_\gamma$ is a $\Gamma$-graded 
commutative $\Q$-algebra
with $\dim_\Q E < \infty$,
then both $E_\sep$ and $\sqrt{0_E}$ are homogeneous.
\end{prop}

\begin{proof}
By Lemma~\ref{orderlem} 
the set $\{ \gamma\in \Gamma : E_\gamma \neq 0 \}$ is finite,
and by Lemma \ref{gammamischomog}  we may assume $\Gamma$ is finite.
For $\sqrt{0_E}$, see Proposition \ref{AGamma}(ii).
For $E_\sep$, the proof is the same.
Namely, suppose  
$\alpha = (\alpha_\gamma)_{\gamma\in \Gamma} \in E_\sep$ 
and let $E' = E\otimes_\Z\Z[\zeta_k]$ with $k$ the
exponent of $\Gamma$.
Then 
$\chi(\delta)^{-1} \in \langle \zeta_k \rangle \subset (E')_\sep$,
and $(E')_\sep$ is a ring that is stable under the ring automorphisms of $E'$.
As in the proof of Proposition~\ref{AGamma},
we obtain $\#\Gamma\cdot\alpha_\delta \in (E')_\sep \cap E = E_\sep$ 
for all $\delta\in\Gamma$.  
Since $(\#\Gamma)^{-1}\in\Q\subset E_\sep$,
we have $\alpha_\delta \in E_\sep$ for all $\delta\in\Gamma$,
as desired.
\end{proof}

\section{Idempotents in graded orders}

In this section we prove Theorem \ref{firstthm}(ii) (see Proposition \ref{idemsequal}).
We will use Proposition \ref{Aorthog} to prove both Theorem \ref{univgradingorderthm} and
Theorem \ref{firstthm}.

Suppose $L$ is a lattice.  If $z\in L$,
then a {\bf decomposition} of $z$ in $L$ is a pair
$(x,y)\in L \times L$ such that
$
z = x + y$ 
and $\langle x,y \rangle \ge 0.
$
We say that such a decomposition is {\bf non-trivial} if $x\neq 0$ and $y \neq 0$.
Call $z$ {\bf indecomposable} (in $L$) if the number of decompositions
of $z$ equals $2$, or equivalently, if 
$z\neq 0$ and $z$ has no non-trivial decompositions.

\begin{rem}
\label{zxyineq}
If $L$ is a lattice and $z = x + y$ with $x,y,z\in L$, then:
\begin{enumerate}
\item
$
\langle x,y \rangle \ge 0 \iff 
\langle z,z \rangle \ge \langle x,x \rangle + \langle y,y \rangle,
$
\item
$
\langle x,y \rangle = 0 \iff 
\langle z,z \rangle = \langle x,x \rangle + \langle y,y \rangle.
$
\end{enumerate}
\end{rem}

\begin{rems}
\begin{enumerate}[leftmargin=*]
\item 
If $z$  
is a shortest non-zero vector in a lattice $L$, then
$z$ is indecomposable.
\item 
If $L$ is a lattice, then $L$ is generated by its set of indecomposable elements.
\end{enumerate}
\end{rems}

Recall that $\Id(A)$ denotes the set of idempotents of a ring $A$.
Below we use the natural lattice structure on a reduced order that was given in 
Example \ref{orderlatticeex4}.

\begin{lem}
\label{xxgesigmalem}
If $A$ is a reduced order and  $x\in A$, 
then $\langle x,x \rangle \ge \#\{ \sigma\in\Rhom(A,\C) : \sigma(x) \neq 0 \}$.
\end{lem}

\begin{proof}
If $\sigma(x)=0$ for
all $\sigma\in\Rhom(A,\C)$, then $x=0$ 
(see for example Lemma 3.1 of \cite{CMorders}),
and the desired result holds.
Assume that $x \neq 0$.
Applying the arithmetic-geometric mean inequality to obtain
the first inequality below, and using that
$\prod_{\sigma(x)\neq 0} \sigma(x)\overline{\sigma(x)} \in\Z_{>0}$
for the second, we have
\begin{multline*}
\langle x,x \rangle 
= \sum_{{\sigma\in\Rhom(A,\C)}\atop{\sigma(x)\neq 0}} \sigma(x)\overline{\sigma(x)}
= \#\{\sigma : \sigma(x)\neq 0 \} \cdot 
\frac{\sum_{\sigma(x)\neq 0} \sigma(x)\overline{\sigma(x)}}{\#\{\sigma : \sigma(x)\neq 0 \}} \\
\ge \#\{\sigma : \sigma(x)\neq 0 \} \cdot 
\left(\prod_{\sigma(x)\neq 0} \sigma(x)\overline{\sigma(x)}\right)^{1/\#\{\sigma : \sigma(x)\neq 0 \}}
\ge \#\{\sigma : \sigma(x)\neq 0 \}.
\end{multline*}
\end{proof}

\begin{lem}
\label{eoneminuse}
If $A$ is a reduced order and  $e\in\Id(A)$, 
then $\langle e,1-e \rangle =  0$.
\end{lem}

\begin{proof}
Since $e\in\Id(A)$, for all $\sigma\in\Rhom(A,\C)$ 
we have $\sigma(e)\in\{ 0,1\}$, 
so $\sigma(e)\overline{\sigma(1-e)}=0$.
Thus,
$
\langle e,1-e \rangle 
= \sum_{\sigma\in\Rhom(A,\C)} \sigma(e)\overline{\sigma(1-e)}
= 0.
$
\end{proof}

\begin{prop}
\label{redorder}
Suppose $A$ is a reduced order.
Then the map
$$
F : \Id(A) \to \{ \text{decompositions of $1$ in $A$} \} 
$$
defined by $e \mapsto (e,1-e)$ is a bijection, and its inverse
sends a  
decomposition $(x,y)$ of $1$ to $x$.
\end{prop}

\begin{proof}
We first show that the map $F$ is well-defined.
Suppose $e\in\Id(A)$. 
By Lemma \ref{eoneminuse} we have $\langle e,1-e \rangle =  0$.
Thus 
$(e,1-e)$ is a 
decomposition of $1$ in $A$, as desired.

The map $F$ is clearly injective.
To see that it is surjective,
suppose  $(x,y)$ is a decomposition of $1$ in $A$.
By Lemma \ref{xxgesigmalem} we have
$\langle x,x \rangle \ge \#\{ \sigma\in\Rhom(A,\C) : \sigma(x) \neq 0 \}$,
and the same with $y$ in place of $x$. 
Using that $x+y=1$ to obtain the third equality, it follows that 
\begin{multline*}
\#\Rhom(A,\C) 
= \rank_\Z A 
= \langle 1,1 \rangle 
\ge \langle x,x \rangle + \langle y,y \rangle \\
\ge \#\{\sigma\in\Rhom(A,\C) : \sigma(x)\neq 0 \} + \#\{\sigma\in\Rhom(A,\C) : \sigma(y)\neq 0 \} \\
= \#\Rhom(A,\C)  + \#\{\sigma\in\Rhom(A,\C) : \sigma(x)\neq 0, \,\, \sigma(y)\neq 0 \} \\ 
= \#\Rhom(A,\C)  + \#\{\sigma\in\Rhom(A,\C) : \sigma(xy)\neq 0 \}. 
\end{multline*}
Thus for all $\sigma\in\Rhom(A,\C)$ we have $\sigma(xy)= 0$.
So $x(1-x) = xy = 0$.
Thus, $x\in \Id(A)$ so $F$ is surjective.
\end{proof}

\begin{cor}
\label{Aconnindecomp}
Suppose $A$ is a reduced order.
Then $A$ is connected if and only if $1$ is indecomposable.
\end{cor}

\begin{lem}
\label{Afacts}
Suppose 
$A$ is a reduced order, $\Gamma$ is a finite abelian group, and 
$(B_\gamma)_{\gamma\in \Gamma}$ is a $\Gamma$-grading of $A$.
Let $k$ denote the exponent of the group $\Gamma$ and  
let $A' = A\otimes_\Z \Z[\zeta_k]$.
Then:
\begin{enumerate}
\item
$A'$ is reduced;
\item
$
\Rhom(A',\C) \cong \Rhom(A,\C) \times \Rhom(\Z[\zeta_k],\C);
$
\item
for all $\alpha,\beta\in A \subset A'$ we have
$
\langle \alpha,\beta \rangle_{A'} =  \varphi(k)\langle \alpha,\beta \rangle_{A},
$
where $\langle \, \, \, ,\, \, \,  \rangle_{A'}$ 
and
 $\langle \, \, \, ,\, \, \,  \rangle_{A}$ are the inner products of
Example \ref{orderlatticeex4} for $A'$ and $A$, respectively.

\end{enumerate}
\end{lem}

\begin{proof}
Part (i) holds since $A'_\Q = A_\Q\otimes_\Q \Q(\zeta_k)$
is a separable algebra over $\Q$ (since $A_\Q$ and $\Q(\zeta_k)$ are).
Part (ii) is immediate.
Part (iii) follows from (ii) since $\#\Rhom(\Z[\zeta_k],\C) = \varphi(k)$,
so each element of $\Rhom(A,\C)$ has $\varphi(k)$ extensions to $A'$.
\end{proof}

\begin{prop}
\label{Aorthog}
Suppose 
$A$ is a reduced order, $\Gamma$ is an abelian group,  
$(B_\gamma)_{\gamma\in \Gamma}$ is a $\Gamma$-grading of $A$,
and
$\langle \, \, \, ,\, \, \,  \rangle$ is the inner product of
Example \ref{orderlatticeex4}.
Suppose $\gamma, \delta \in \Gamma$ and $\gamma \neq \delta$.
Then $\langle B_\gamma,B_\delta \rangle =0$.
\end{prop}

\begin{proof}
The conclusion is clear if $B_\gamma =0$ or $B_\delta =0$.
Thus, we can (and do) replace $\Gamma$ by the subgroup
$\langle \gamma\in \Gamma : B_\gamma \neq 0 \rangle$,
which is finite by Proposition \ref{redordercor}(i).

Let $k$ denote the exponent of the group $\Gamma$
and embed 
$A$ in $A' = A[\zeta_k] 
= \bigoplus_{\gamma\in \Gamma} B'_\gamma
$
where $B'_\gamma = B_\gamma\otimes_\Z \Z[\zeta_k]$.
It suffices to show $\langle B'_\gamma,B'_\delta \rangle_{A'} =0$.
Let $\alpha\in B'_\gamma$ and $\beta\in B'_\delta$.
Choose $\chi\in\hat{\Gamma}_k$ such that $\chi(\gamma)\neq \chi(\delta)$.
Since  
$\chi$ acts on $A'$ by a ring automorphism (Lemma \ref{fchithm}) we have
$$
\langle \alpha,\beta \rangle_{A'} =  
\langle \chi\ast(\alpha),\chi\ast(\beta) \rangle_{A'} =  
\langle \chi(\gamma)\alpha,\chi(\delta)\beta \rangle_{A'} =  
\langle \alpha,\chi(\gamma)^{-1}\chi(\delta)\beta \rangle_{A'}.
$$
Thus,
\begin{equation}
\label{deltagammaeqn}
\langle B'_\gamma,(1 - \chi(\gamma)^{-1}\chi(\delta))B'_\delta \rangle_{A'} = 0.
\end{equation}
We have $\chi(\gamma)^{-1}\chi(\delta)\in \langle \zeta_k \rangle \smallsetminus \{ 1\}$.
Thus, $1 - \chi(\gamma)^{-1}\chi(\delta)$ divides $\prod_{i=1}^{k-1}(1-\zeta_k^i)=k$
in $\Z[\zeta_k]$.
By \eqref{deltagammaeqn} we now have 
$0 
= \langle B'_\gamma,kB'_\delta \rangle_{A'} 
= k\langle B'_\gamma,B'_\delta \rangle_{A'}$.
Thus, $\langle B'_\gamma,B'_\delta \rangle_{A'}=0$.
\end{proof}

\begin{prop}
\label{idemsequal}
Suppose 
$A$ is an order, $\Gamma$ is an abelian group, and 
$(B_\gamma)_{\gamma\in \Gamma}$ is a $\Gamma$-grading of $A$.
Then $\Idem(A) = \Idem(B_1)$,
and $A$ is 
connected if and only if $B_1$ is connected.
\end{prop}

\begin{proof}
The inclusion $\Idem(B_1)\subset\Idem(A)$ is clear.
For the reverse inclusion, take $e=(e_\gamma)_{\gamma\in \Gamma}\in \Idem(A)$.

We first assume $A$ is reduced.
By Lemma \ref{B1ring}(i) we have 
$(1-e)_\gamma = -e_\gamma$ if $\gamma\neq 1$, and $(1-e)_1 = 1-e_1$.
By Lemma \ref{eoneminuse} and Proposition \ref{Aorthog} we have 
$$
0 = \langle e,1-e \rangle =  
\sum_{\gamma\in \Gamma} \langle e_\gamma,(1-e)_\gamma \rangle =   
\langle e_1,1-e_1 \rangle - \sum_{\gamma \neq 1} \langle e_\gamma,e_\gamma \rangle 
\leq  \langle e_1,1-e_1 \rangle,
$$
so $(e_1,1-e_1)$ is a decomposition of $1$.
Now Proposition \ref{redorder} and Lemma \ref{eoneminuse} give
$\langle e_1,1-e_1 \rangle = 0$ so 
$0 = \sum_{\gamma \neq 1} \langle e_\gamma,e_\gamma \rangle $,
and all $e_\gamma$ with $\gamma \neq 1$ are $0$.
Hence $e \in B_1$.

For the general case, the natural maps 
$
\Idem(A) \to \Idem(A/\sqrt{0_A})
$
and
$
\Idem(B_1) \to \Idem(B_1/\sqrt{0_{B_1}})
$
are bijections (this follows, for example, from Theorem 1.5 of \cite{Qalgs}).
By the reduced case, the natural map
$\Idem(B_1/\sqrt{0_{B_1}}) \to \Idem(A/\sqrt{0_A})$ is a bijection.
It follows that the inclusion 
$\Idem(B_1) \hookrightarrow \Idem(A)$ is a bijection.
In particular, $A$ is 
connected if and only if $B_1$ is connected.
\end{proof}

\section{Roots of unity in graded orders}

In this section we prove Theorem \ref{firstthm}(iii).

\begin{rem}
\label{exampremrtunity}
If $A$ is a reduced order with a $\Gamma$-grading
and $\zeta = (\zeta_\gamma)_{\gamma\in \Gamma}\in\mu(A)$, then
by  \eqref{zetarkeqn} and
Proposition \ref{Aorthog} we have
$\rank(A) = \langle\zeta,\zeta\rangle =\sum_\gamma \langle\zeta_\gamma,\zeta_\gamma\rangle$.
If each non-zero term in the latter sum were at least $\rank(A)$, then there 
would be at most one such term, and Theorem~\ref{firstthm}(iii) 
would follow.
However, Example \ref{orderlatticeex7} exhibits a connected reduced order $A$ and 
$x\in A$ with $0 < \langle x,x\rangle < \rank(A)$. Thus,
more is required to prove Theorem~\ref{firstthm}(iii).
\end{rem}

\begin{lem}
\label{indecomponepiece}
If 
$A$ is a reduced order, $\Gamma$ is an abelian group,
$(B_\gamma)_{\gamma\in \Gamma}$ is a $\Gamma$-grading of $A$, 
and $\alpha\in A$ is indecomposable, then
there exists $\delta\in \Gamma$ such that $\alpha\in B_\delta$.
\end{lem}

\begin{proof}
Pick $\delta\in \Gamma$ with $\alpha_\delta \neq 0$.
Then $\alpha = \alpha_\delta + (\alpha - \alpha_\delta)$,
and we have $\alpha_\delta \in B_\delta$ and
$\alpha - \alpha_\delta \in \bigoplus_{\gamma\neq\delta} B_\gamma$,
so $\langle \alpha_\delta,\alpha - \alpha_\delta \rangle = 0$ 
by Proposition \ref{Aorthog}.
Since 
$(\alpha_\delta,\alpha - \alpha_\delta)$ cannot be a non-trivial decomposition of
the indecomposable element $\alpha$,
we have $\alpha - \alpha_\delta = 0$ as desired.
\end{proof}

\begin{prop}
\label{connhomogprop}
If 
$A$ is an order, $\Gamma$ is an abelian group,
$(B_\gamma)_{\gamma\in \Gamma}$ is a $\Gamma$-grading of $A$,
and $B_1$ is connected, 
then  
$\mu(A) \subset \bigcup_{\gamma\in \Gamma} B_\gamma$.
\end{prop}

\begin{proof}
Proposition \ref{idemsequal} shows that $A$ is connected.
Take $\zeta = (\zeta_\gamma)_{\gamma\in \Gamma}\in\mu(A)$. 

First suppose $A$ is reduced.
Then $1$ is indecomposable in $A$ by Corollary \ref{Aconnindecomp}.
The map $x\mapsto \zeta x$ is a lattice
automorphism of $A$.
Hence $\zeta$ is also indecomposable in $A$.
By Lemma \ref{indecomponepiece}, 
there exists $\delta\in \Gamma$ such that $\zeta\in B_\delta$, as desired.

For the general case, applying Proposition \ref{Esephomog}
to $E=A_\Q$ shows that $\zeta_\gamma \in E_\sep$ for all $\gamma\in \Gamma$.
Also, 
$\zeta\bmod \sqrt{0_A} \in A/\sqrt{0_A} = 
\bigoplus_{\gamma\in \Gamma} B_\gamma/(\sqrt{0_A} \cap B_\gamma)$
is a root of unity, so by the reduced case there is a unique
$\delta\in \Gamma$ such that $(\zeta\bmod \sqrt{0_A})_\delta$
is a root of unity and for all $\gamma\neq\delta$ we have 
$0 = (\zeta\bmod \sqrt{0_A})_\gamma = \zeta_\gamma\bmod (\sqrt{0_A}\cap B_\gamma)$.
Thus for all $\gamma\neq\delta$ we have 
$\zeta_\gamma \in \sqrt{0_E} \cap E_\sep = \{ 0\}$.
\end{proof}

\section{Universal gradings---lemmas and examples}
\label{lemexsect}

The results in this section follow in a straightforward
way from the definitions, and are left as exercises.

\begin{lem}
Suppose $A$ is a  
ring and $\Gamma$ is an abelian group.
\begin{enumerate}
\item
Suppose 
$\B = (B_\gamma)_{\gamma\in \Gamma}$ is a $\Gamma$-grading of $A$,
suppose $\Delta$ is an abelian group, 
suppose $f : \Gamma \to \Delta$ is a
group homomorphism, and let
$f_\ast(\B) = 
(\sum_{\gamma\in f^{-1}(\delta)}B_\gamma)_{\delta\in \Delta}.
$
Then  
$f_\ast(\B)$ is a $\Delta$-grading of $A$.
\item
The map $\Gamma \mapsto \{ \text{$\Gamma$-gradings of $A$} \}$
is a covariant functor from the category of abelian groups to
the category of sets.
\end{enumerate}
\end{lem}

An  abelian group $H$ is called {\bf indecomposable} if $H \neq 1$ and
whenever $H = H_1 \oplus H_2$ with abelian groups $H_1$ and $H_2$
then $H_1 =1$ or $H_2 =1$.

\begin{lem}
\label{univlem}
Suppose $A$ is a  
ring.
\begin{enumerate}
\item
If $(\Gamma_1, (B_\gamma)_{\gamma\in \Gamma_1})$ and 
$(\Gamma_2,(C_\gamma)_{\gamma\in \Gamma_2})$
are  universal gradings of $A$, 
then there is a unique group isomorphism 
$\sigma : \Gamma_1 \to \Gamma_2$ such that for all
$\gamma\in\Gamma_1$ we have $B_\gamma = C_{\sigma(\gamma)}$.
\item
If 
$(\Gamma,(A_\gamma)_{\gamma\in \Gamma})$ is  a universal grading of $A$, and $(C_\delta)_{\delta\in \Delta}$ is
a  $\Delta$-grading of $A$, then for each $\delta\in \Delta$
for which $C_\delta$ is an indecomposable abelian group
there exists $\gamma\in \Gamma$ with $C_\delta = A_\gamma$.
\end{enumerate}
\end{lem}

\begin{exs} 
\label{nounivgradexs}
We leave verifications of the below statements as an exercise.
A hint is to use Lemma \ref{univlem}(ii).
\begin{enumerate}
\item
The cyclotomic field $\Q(\zeta_8)$ has
a $\Z/4\Z$-grading $\bigoplus_{j=0}^3 \Q\cdot\zeta_8^j$
and a $(\Z/2\Z\times \Z/2\Z)$-grading 
$\Q \oplus \Q \mathrm{i} \oplus \Q \sqrt{2} \oplus \Q \mathrm{i}\sqrt{2}$
and has no universal grading.
For $t\ge 4$, the field $\Q(\zeta_{2^t})$ equals $\Q(\eta)$, where  
$\eta=\zeta_{2^t}\sqrt{2}$,
it has the two gradings 
$\bigoplus_{j=0}^{2^{t-1}-1}\Q\cdot\zeta_{2^t}^j$ and 
$\bigoplus_{j=0}^{2^{t-1}-1}\Q\cdot\eta^j$ 
by a cyclic group of order $2^{t-1}$, and it has no universal grading. This example is taken from \cite{DvG}.
\item
The field $\Q(\sqrt[3]{2},\zeta_3)$ has three different $\Z/6\Z$-gradings
in which all pieces have dimension one over $\Q$, and has no universal grading.
\item
A $\Z/2\Z$-grading of $\F_{5^6}$ is $\F_{5^3} \oplus \F_{5^3}\cdot\sqrt{2}$,
a $\Z/3\Z$-grading of $\F_{5^6}$ is 
$\F_{5^2} \oplus \F_{5^2}\cdot\zeta_9 \oplus \F_{5^2}\cdot\zeta_{9}^2$,
but $\F_{5^6}$ has no universal grading.
\item
If $d \in \Z$ and $d$ is not a square, 
then the $\Z/2\Z$-grading $\Z \oplus \sqrt{d}\Z$
is the universal grading on $\Z[\sqrt{d}]$.
If $A$ is an order of rank $2$ and odd discriminant, then
the grading by the trivial group is the universal grading on $A$.
\item
The ring $\Z[\sqrt[3]{2},\zeta_3]$ has a universal grading 
$\bigoplus_{j=0}^2 \Z[\zeta_3]\sqrt[3]{2}^j$ by a cyclic group of order $3$.
\item
The ring $\Z[X]/(X^2) = \Z[\varepsilon]$ has a universal grading by an infinite cyclic group $\Gamma= \langle c \rangle$, with $\Z[\varepsilon]_1=\Z$, and 
$\Z[\varepsilon]_c = \Z\varepsilon$, and  $\Z[\varepsilon]_\gamma = 0$ 
for all $\gamma \in \Gamma \smallsetminus \{ 1,c\}$.
This also gives a $\Z/n\Z$-grading on the ring for every $n\in\Z_{>1}$.
This non-reduced graded order has no universal grading by a
finite abelian group.
\item
Let $A$ be the subring of $\Z[X]/(X^4)$ generated by the images of 
$1$, $2X(1+X)$, and $2X^2(1+X)$.
Then $A$ is a non-reduced order, and the grading of $A$ by the trivial group 
is the universal grading of $A$.
\item
The ring  
$\Z[X,Y]/(X,Y)^2 = \Z[\varepsilon, \eta]$,
with $\varepsilon = X \bmod (X,Y)^2$ and $\eta = Y \bmod (X,Y)^2$, has 
no universal grading. 
If $\Gamma$ is any group, and $\sigma$ and $\tau$ are non-identity distinct
elements of $\Gamma$, then one grading is given by
$B_1 =\Z$, $B_\sigma =\Z\varepsilon$, $B_\tau = \Z\eta$
and another by 
$B_1 =\Z$, $B_\sigma =\Z(\varepsilon + \eta)$, $B_\tau = \Z(\varepsilon + 2\eta)$.
\item
If $\Gamma$ is an abelian group, then 
the universal grading of the group ring $\Z[\Gamma]$
is the natural $\Gamma$-grading $(\Z\cdot\gamma)_{\gamma\in \Gamma}$.
\end{enumerate}
\end{exs}

\section{$S$-decompositions of lattices}
\label{Sdecompsect}

We give a result on $S$-decompositions of lattices 
that we will use 
in \S \ref{Thm11pf} to prove Theorem \ref{univgradingorderthm}.

If $L$ is a lattice and
$S$ is a set, then an
{\bf $S$-decomposition} of $L$ is a system 
$(L_s)_{s\in S}$ of subgroups of $L$ such that:
\begin{enumerate}
\item
if $s,t\in S$ and $s\neq t$, then $\langle L_s,L_t\rangle =0$, and
\item
$\sum_{s\in S} L_s = L$.
\end{enumerate}
This implies that
$L = \bigoplus_{s\in S} L_s$,
in the sense that the map  
$\bigoplus_{s\in S} L_s \to L, 
(\alpha_s)_{s\in S} \mapsto \sum_{s\in S}\alpha_s$
is bijective.

An $S$-decomposition $(L_s)_{s\in S}$ of a lattice $L$ is
{\bf universal} if for every set $T$ and
every $T$-decomposition $(M_t)_{t\in T}$ of $L$, 
there is a unique map $f : S \to T$ such that
for all $t\in T$ we have $M_t = \sum_{s\in f^{-1}(t)} L_s$.

If a set $S$ and a universal $S$-decomposition exist for a given lattice, then by a standard argument $S$ and that decomposition are, in an obvious sense, unique up to a unique isomorphism.

\begin{thm}
\label{EicherThm}
Every lattice has a  
universal $S$-decomposition for some 
finite set $S$, and for that universal $S$-decomposition all $L_s$ are non-zero.
\end{thm}

Theorem \ref{EicherThm} is classical
and due to Eichler,
and can be easily proved using
the proof of Theorem 6.4 on p.~27 of \cite{Milnor}.

\section{Proof of Theorem \ref{univgradingorderthm}}
\label{Thm11pf}
We now prove Theorem \ref{univgradingorderthm}.
Since $A$ is a reduced order, it has a lattice structure as in Example~\ref{orderlatticeex4}.
By Theorem \ref{EicherThm} the lattice $A$ has a
universal $S$-decomposition $A = \bigoplus_{s\in S} L_s$ for some
finite 
set $S$, and each $L_s$ is non-zero.
Let $\Gamma$ be the abelian 
group with generating set $S$ and
relations $s_1 \cdot s_2 = s_3$ whenever there are $x\in L_{s_1}$ and
$y\in L_{s_2}$ such that when we write
$xy = \sum_{s\in S} z_s$ with $z_s \in L_s$ we have $z_{s_3} \neq 0$.
This produces a   
group $\Gamma$ equipped with a map 
$\h : S \to \Gamma$, $s \mapsto s$, and we obtain a 
$\Gamma$-decomposition 
$(B_\gamma)_{\gamma\in \Gamma}$ of $A$ with
$B_\gamma = \sum_{s\in \h^{-1}(\gamma)} L_s$.
If $s_1\in \h^{-1}(\gamma_1)$ and ${s_2}\in \h^{-1}(\gamma_2)$
with $\gamma_1,\gamma_2\in\Gamma$, then
$$
L_{s_1}\cdot L_{s_2} \subset  
\sum_{u\in S, u={s_1}{s_2}}L_u
 \subset \sum_{u\in \h^{-1}(\gamma_1\gamma_2)} L_u = B_{\gamma_1\gamma_2}.
$$
Thus $B_{\gamma_1}B_{\gamma_2} \subset B_{\gamma_1\gamma_2}$, 
so the $\Gamma$-decomposition $\B = (B_\gamma)_{\gamma\in \Gamma}$ 
is a $\Gamma$-grading.

Since 
each $L_s$ is non-zero,
we have that $B_\gamma \neq 0$ for all $\gamma \in h(S)$, so
$\Gamma \supset \langle \gamma\in\Gamma : B_\gamma \neq 0\rangle
\supset \langle h(S) \rangle \supset \Gamma$.
It now follows from Proposition \ref{redordercor}(i) that $\Gamma$ is finite.

To show the $\Gamma$-grading
$\B$  
is universal, let $\CC =(C_\delta)_{\delta\in \Delta}$
be a $\Delta$-grading of $A$, with $\Delta$ an abelian group.
By Proposition \ref{Aorthog}, 
we have that $\CC$ is a $\Delta$-decomposition of the lattice $A$,
so there is a unique map $g : S \to \Delta$
such that for all $\delta\in \Delta$ we have 
$C_\delta = \sum_{s\in g^{-1}(\delta)} L_s$.
If ${s_1}{s_2}=u$ is one of the relations for the group $\Gamma$,
then for some $x\in L_{s_1} \subset C_{g({s_1})}$ and $y\in L_{s_2} \subset C_{g({s_2})}$
we have a product $xy$ with $L_u$-coordinate non-zero,
so with $C_{g(u)}$-coordinate non-zero.
But $C_{g({s_1})}C_{g({s_2})} \subset C_{g({s_1})g({s_2})}$ so 
$g(u) = g({s_1})g({s_2})$.
So there is a unique group homomorphism $f : \Gamma \to \Delta$
such that $f\circ \h = g$.
This implies that 
$f_\ast\B =\CC,
$
so the map $f \mapsto f_\ast\B$ is surjective.
To show it is  injective, suppose $\tilde{f} : \Gamma \to \Delta$
is a  group homomorphism such that
$\tilde{f}_\ast\B =\CC$.
By the uniqueness of $f$ we have 
$f\circ\h = \tilde{f}\circ\h$.
Since $\Gamma = \langle \h(S)\rangle$ it follows that
$f = \tilde{f}$, so the map $f \mapsto f_\ast\B$ is injective.

\section{Proof of Theorem \ref{noncyclicthm}}
\label{noncyclicthmpf}

\begin{lem}
\label{noncycliclem}
Suppose $E=\bigoplus_{\gamma \in \Gamma} D_\gamma$ is a finite \'etale $\Q$-algebra graded by a finite abelian group $\Gamma$, suppose 
$\Gamma = \langle \gamma\in \Gamma : D_\gamma \neq 0 \rangle$,
and suppose $D_1$ is a field.
Then $\dim_{D_1}D_\gamma=1$ for all $\gamma \in \Gamma$. 
\end{lem}

\begin{proof}
Each non-zero homogeneous element has a power in $D_1$.
That power is non-zero, hence a unit.
Thus all homogeneous elements are units.
If $\gamma\in \Gamma$ and $0 \neq x\in D_\gamma$, then the map
$D_1 \to D_\gamma$, $a \mapsto ax$ is an isomorphism of
$D_1$-vector spaces.

To see that each $D_\gamma$ is non-zero, take $\gamma \in \Gamma$
and write it as $\gamma = \prod_{i=1}^r \gamma_i$ with each $D_{\gamma_i} \neq 0$.
For each $i$, choose $0 \neq x_i \in D_{\gamma_i}$.
Then $0 \neq \prod_{i=1}^r x_i \in D_{\gamma}$.
\end{proof} 
 
\begin{lem}
\label{noncycliclem2}
Suppose $A$ is a Dedekind order and 
$A = \bigoplus_{\gamma \in \Gamma} B_\gamma$ is  a $\Gamma$-grading.
Then the order $B_1$ is also Dedekind.
\end{lem}

\begin{proof}
We have $B_1 = A \cap (B_1)_\Q$. It follows that 
$B_1$ is the ring of integers of the number field $(B_1)_\Q$.
\end{proof}

Next we prove Theorem \ref{noncyclicthm}.
It suffices to prove that if $p$ is prime and $A = \bigoplus_{\gamma \in \Gamma} B_\gamma$ is a Dedekind order graded by a finite abelian $p$-group $\Gamma$ with each $(B_\gamma)_\Q$ one-dimensional over the field $(B_1)_\Q$, then $\Gamma$ is cyclic. 
To see that this suffices, invoke Lemma \ref{gammamisc}, replace $\Gamma$ by its $p$-primary component (viewing that component either as a subgroup or as a quotient group), and apply  
Lemma \ref{noncycliclem} with $E=A_\Q$.

Let $p$, $A$, $\Gamma$, $(B_\gamma)_{\gamma \in \Gamma}$ be as above 
and let $q$ be the exponent of the $p$-group $\Gamma$. 
By Lemma \ref{noncycliclem2} we have that $B_1$ is a Dedekind order.

Let $\pp$ be a prime ideal of $B_1$ containing $p$. Define the  
ring homomorphism $\phi: A \to A/\pp A$ by $\phi(x)=(x^q + \pp A)$; this is the canonical map $A \to A/\pp A$ followed by the $q$-th powering map from $A/\pp A$ to itself, the latter being a ring homomorphism because $A/\pp A$ contains the finite field $B_1/\pp$ of characteristic $p$. 
The restriction of $\phi$ to $B_1$ is the canonical map $B_1 \to B_1/\pp$ followed by an automorphism of $B_1/\pp$. For each $\gamma \in \Gamma$ one has $(B_\gamma)^q \subset B_1$, so $\phi(B_\gamma)$ lands in the subring $B_1/\pp$ of $A/\pp A$. Since the $B_\gamma$ generate $A$,  the image of $\phi$ in fact lies in $B_1/\pp$, giving the following diagram.
$$
\xymatrix@R1pt{
A \ar^{\phi}[r] \ar[ddr] & A/\pp A \\
\cup & \cup \\
B_1 \ar@{->>}[r] & B_1/\pp
}
$$
Let $\rr = \ker \phi$. Then $\rr$ is a prime ideal of $A$ with $A/\rr \isom B_1/\pp$, 
so $\rr$ lies over $\pp$ with residue class field degree $f(\rr/\pp)=1$. 
Now we consider the familiar formula 
\begin{equation}
\label{familareqn}
\sum_\qq e(\qq/\pp)f(\qq/\pp) = [A_\Q:(B_1)_\Q] = \#\Gamma,
\end{equation}
the sum ranging over the prime ideals $\qq$ of $A$ lying over $\pp$ and $e(\qq/\pp)$ denoting the ramification index; the last equality follows from our assumption on the $(B_\gamma)_\Q$. 
Let $\qq$ be one of those prime ideals. 
For each $x \in \rr$ one has $x^q \in \pp A \subset \qq$, so $x \in \qq$.
This proves $\rr \subset \qq$, hence $\rr = \qq$, since $\rr$ is maximal. 
Thus there is only one $\qq$, namely $\qq = \rr$.  
Formula \eqref{familareqn} 
now becomes $e(\rr/\pp) = \#\Gamma$. 
For each $x \in \rr$ one has $x^q \in \pp A = \rr^{e(\rr/\pp)} = \rr^{\#\Gamma}$, so $q\cdot\ord_\rr(x) \ge \#\Gamma$; here $\ord_\rr$ counts factors $\rr$. 
Picking $x \in A$ such that $\ord_\rr(x) = 1$, then $x \in \rr$ so $q \ge \#\Gamma$. But a finite abelian 
group whose exponent is at least its order is clearly cyclic.
This gives the desired result.

\begin{rem}
Note that instead of requiring that $A$ be Dedekind, it suffices that it be locally Dedekind at all primes dividing its $\Z$-rank.
\end{rem}

\end{document}